\documentclass{amsart}
\usepackage{latexsym}
\usepackage{amsfonts}
\usepackage{amscd}
\usepackage{array}
\usepackage{amsmath}
\usepackage{epsfig}
\usepackage{amssymb}
\usepackage{amsthm}
\usepackage[colorlinks,
            linkcolor=black, 
            citecolor=black,
            filecolor=black,
            urlcolor=black, 
            anchorcolor=black,
            menucolor=black,
            bookmarksopen=false,
            bookmarks=false, 
            breaklinks=true]{hyperref}

\newtheorem{theorem}{Theorem}[section]
\newtheorem{lemma}[theorem]{Lemma}
\newtheorem{prop}[theorem]{Proposition}
\newtheorem{conj}[theorem]{Conjecture}
\newtheorem{corollary}[theorem]{Corollary}

\theoremstyle{definition}
\newtheorem{definition}[theorem]{Definition}
\newtheorem{remark}[theorem]{Remark}
\newtheorem*{remark*}{Remark}

\def\vol{{\rm vol}}
\def\R{\mathbb{R}}
\def\Z{\mathbb{Z}}

\def\Q{\mathbb{Q}}
\def\P{\mathbb{P}}
\def\det{{\rm det}}
\newcommand{\pro}[2]{\langle #1, #2 \rangle}
\DeclareMathOperator{\conv}{conv}

\newcommand{\0}{\mathbf{0}}
\newcommand{\1}{\mathbf{1}}

\title{On the equality case in Ehrhart's volume conjecture}

\author{Benjamin Nill}
\address{Benjamin Nill \\ Case Western Reserve University \\  Cleveland, OH \\ USA}
\email{benjamin.nill@case.edu}
\thanks{BN is supported by the US National Science Foundation (DMS 1203162)}

\author{Andreas Paffenholz}
\address{Andreas Paffenholz \\ TU Darmstadt \\ Darmstadt\\ Germany}
\email{paffenholz@mathematik.tu-darmstadt.de}
\thanks{AP is supported by the  German Research Foundation (DFG SPP 1489)}

\begin{document}
\maketitle

\begin{abstract}
  Ehrhart's conjecture proposes a sharp upper bound on the volume of a convex body whose barycenter
  is its only interior lattice point. Recently, Berman and Berndtsson proved this conjecture for a class of
  rational polytopes including reflexive polytopes. In particular, they showed that the complex
  projective space has the maximal anticanonical degree among all toric K\"ahler-Einstein Fano
  manifolds.

  In this note, we prove that projective space is the only such toric mani\-fold with maximal degree
  by proving its corresponding convex-geometric statement. We also discuss a generalized version of
  Ehrhart's conjecture involving an invariant corresponding to the so-called greatest lower bound on
  the Ricci curvature.
\end{abstract}

\section{Introduction}

Minkowski's famous lattice point theorem gives a sharp upper bound on the volume of a
centrally-symmetric convex body containing only the origin as a strictly interior lattice
point~\cite{GW93}. A \emph{convex body} $K\subseteq \R^n$ is a compact convex set. $K$ is
\emph{centrally symmetric} if $-x\in K$ for all $x\in K$. Two convex bodies $K,K'$ in $\R^n$ are
\emph{unimodularly equivalent}, if there is an affine lattice automorphism of $\Z^n$ mapping $K$
onto $K'$. Note that this transformation does not change the volume. Let $\Delta_n$ be the {\em
  $n$-dimensional unimodular simplex} with vertices $\0,e_1,\ldots,e_n$, where $\0=(0,\ldots,0)$, and $e_1, \ldots, e_n$ is
the standard lattice basis of $\Z^n$.

In 1964 Ehrhart \cite{Ehr64} posed the following generalization of Minkowski's theorem as a
conjecture (see also \cite[E13]{CFG91} and \cite{GW93}).
\begin{conj}[Ehrhart '64] Let $K \subseteq \R^n$ be an $n$-dimensional convex body with barycenter
  $\0$. If $K$ contains only the origin as an interior lattice point, then
\[\vol(K) \leq (n+1)^n / n!,\]
where equality holds if and only if $K$ is unimodularly equivalent to $(n+1) \Delta_n$.
\label{ehrhart}
\end{conj}
Ehrhart proved the upper bound in the conjecture in several special cases, e.g., in dimension $2$
\cite{Ehr55a} and for simplices \cite{Ehr78}. The uniqueness of the equality case is suggested in
\cite{Ehr64} and proved in dimension $2$ \cite{Ehr55b}.

A very good upper bound which is only slightly weaker than the conjectured one can be deduced from a
paper of Milman and Pajor. In Corollary 3(2) of \cite{MP00} they show that $\vol(K)\le 2^n\vol(K\cap
(-K))$ for any compact convex body with barycenter in the origin. If additionally the origin is the
only lattice point in the interior of $K$, then we can combine this result with Minkowski's first
theorem to obtain $\vol(K)\le 4^n$. Let us note that by Stirling's approximation the conjectural
upper bound is asymptotically equal to $e^{n+1}/\sqrt{2\pi n}$.

For rational polytopes (i.e., polytopes with vertices in $\Q^n$), there is a natural relation of
this conjecture to differential geometry via toric varieties, see e.g., \cite{Mab87,NP11}. A
polytope $Q$ is a \emph{lattice polytope} if its vertices are in $\Z^n$.  Let the origin $\0$ be an
interior point of $Q$. In this case, the {\em dual polytope} $Q^*$ is defined as
\[Q^*\ :=\ \left\{x \in \R^n \;:\; \pro{v}{x} \ge -1 \; \text{ for all vertices } v \text{ of }
  Q\right\},\] %
where we chose an inner product $\pro{\cdot}{\cdot}$ on $\R^n$ (We remark that the usual convention
is to consider dual vector spaces. We use the above notation for the sake of simplicity of this note.)
It follows from $Q$ being a lattice polytope that $\0$ is the only interior lattice point of the
rational polytope $Q^*$. In general, $Q^*$ is not a lattice polytope.  If it is, then we say that
$Q$ and $Q^*$ are \emph{reflexive polytopes}.

For the following, let $Q\subseteq \R^n$ be a \emph{Fano polytope}, i.e., an $n$-dimensional lattice
polytope that contains the origin in its interior and whose vertices are primitive (i.e., not
multiples of a non-zero lattice point).  Then the associated fan over the faces of $Q$ gives rise to
a {\em toric Fano variety} $X$ (see, \textit{e.g.},~\cite{Fulton} for details of this
correspondence). In particular, $X\cong \P^n$ if and only if the dual of the associated polytope $Q$
is unimodularly equivalent to $(n+1)\Delta_n$. Wang and Zhu~\cite{WZ04} (in the nonsingular case)
and Berman and Berndtsson~\cite{BB12} (in the singular case) showed that $X$ admits a
K\"ahler--Einstein metric if and only if the barycenter of $Q^*$ is $\0$. Moreover, it is known that
the anticanonical degree $c_{1}(X)^{n}$ equals $n!\; \vol(Q^*)$. Hence, Ehrhart's conjecture can be
phrased in this setting as the problem of finding the maximal degree of a K\"ahler-Einstein toric
Fano variety \cite{NP11}. Recently, this was solved by Berman and Berndtsson \cite[Thm.1.2]{BB12}:

\begin{theorem}[Berman, Berndtsson '12]
  Let $X$ be an $n$-dimensional toric Fano variety which admits a (singular) K\"ahler-Einstein
  metric.  Then its first Chern class $c_{1}(X)$ satisfies the following upper bound:
\[
c_{1}(X)^{n}\leq(n+1)^{n}
\]
\label{thm-bb}
\end{theorem}

The motivation of this note is to clarify the equality case:

\begin{prop}\label{main-alg-geo}
  In the situation of Theorem~\ref{thm-bb}, $c_{1}(X)^{n}=(n+1)^{n}$ if and only if $X \cong \P^n$.
\end{prop}
This is a special case of the following main result. 
\begin{theorem}
  The conclusion of Conjecture~\ref{ehrhart} holds for any $n$-dimensional convex body $K \subseteq \R^n$ with
  barycenter $\0$ that is contained in the dual polytope of a lattice polytope.
\label{main}
\end{theorem}
The proof relies on the convex-geometric part of the argument in \cite{BB12} which Berman and
Berndtsson attribute to Bo'az Klartag. The only novel part is the treatment of the equality case.
Note that the condition in the theorem is equivalent to the condition that the dual of $K$ contains
a full-dimensional lattice polytope having the origin in its interior. Here is an example of a
convex body that does not satisfy this condition. Let $K :=
\conv\left(\pm(3/2,1/4),\pm(3/2,5/4)\right)$. Then $K$ is a convex body whose only interior lattice
point is $\0$. Its dual is $K^*=\conv\left(\pm(1, -2), \pm(\frac23,0)\right)$, and the only lattice
points in $K^*$ are $\0$ and $\pm(1,-2)$.

Theorem~\ref{main} generalizes Corollary~1.4 in \cite{BB12}. In particular, Ehrhart's conjecture
holds for all reflexive polytopes with barycenter $\0$. This huge class of lattice polytopes is of importance
in string theory since they give rise to many examples of mirror-symmetric Calabi-Yau manifolds
\cite{Bat94, Nil05,KS98,KS00}.

\smallskip

\textbf{Organization of the paper:} Section~2 contains the proof of Theorem~\ref{main}. In Section~3
we discuss a variant of Conjecture~\ref{ehrhart} regarding the so-called greatest lower bound on the Ricci
curvature.

\smallskip

\textbf{Acknowledgments:} We would like to thank Robert Berman and Xiaodong Wang for their interest
and for telling us about the generalized version of Ehrhart's conjecture. We thank the referee for
several suggestions to improve the paper. In particular, he pointed us to the result of Milman and
Pajor mentioned above and its relation to Ehrhart's conjecture. We are especially grateful to Stanislaw Szarek for providing the alternative proof in Remark~\ref{szarek} and publishing his proof of Gr\"unbaum's inequality as a courtesy to the readers.

\section{Proof of Theorem~\ref{main}}

We may assume $n\ge 2$. Let $\0$ be the origin in $\R^n$ and $\1:=(1,\ldots,1)$.  We follow the
approach in the proof of Corollary~1.4 and Remark~3.2 in \cite{BB12}. A key ingredient of the proof
is Gr\"unbaum's inequality~\cite{Gru60}. For any convex body $K\subseteq\R^n$ and any closed
half-space $H$ that contains the barycenter of $K$ it states that
\begin{align}
  \label{eq:1}
  \vol(K\cap H)\ \ge\ \left(\frac n{n+1}\right)^n\vol(K)\,.
\end{align}

Now let $K \subseteq \R^n$ be an $n$-dimensional convex body with barycenter $\0$ that is contained
in $P := Q^*$ for a lattice polytope $Q$ containing $\0$ in its interior. Fix a vertex $v$ of $P$.
The facets $F$ of $P$ are in one-to-one correspondence with the vertices $l_F$ of $Q$.  Choose
facets $F_1, \ldots, F_n$ containing $v$ such that $l_{F_1}, \ldots, l_{F_n}$ span $\R^n$.  We
consider the invertible affine-linear map
\begin{align*}
  \phi \;:\; \R^n\ &\longrightarrow\ \R^n\; \\
  x\ &\longmapsto (\pro{l_{F_1}}{x}+1, \ldots, \pro{l_{F_n}}{x}+1). 
\end{align*}
By construction, $\phi$ maps the polytope $P$ into the non-negative orthant $\R^n_{\ge 0}$.  Since
$l_{F_1}, \ldots, l_{F_n}$ are lattice points we have $|\det(\phi)| \ge 1$. In particular,
\begin{equation}
\vol(K)\ \le \ |\det(\phi)|\vol(K)\ =\ \vol(\phi(K))\,.\label{det}
\end{equation}
The barycenter of $\phi(K)$ is the image of the barycenter of $K$, so $b:=\phi(\0)=\1$. Consider 
$\eta:=\1$ as a linear functional on $\R^n$. Then $\eta^- := \{x \in \R^n \;:\; \pro{\eta}{x} \le
n\}$ is an affine half-space containing the origin in its interior and $b$ on the boundary. Using
$\phi(K) \subseteq \phi(P) \subseteq \R^n_{\ge 0}$, we get
\begin{align}
  \label{eq:2}
  \begin{split}
    \left(\frac{n}{n+1}\right)^n \vol(\phi(K))\  &\le\ \vol(\phi(K) \cap \eta^-)\ \le\ \vol(\phi(P) \cap \eta^-)\\
    &\le\ \vol(\R^n_{\ge 0} \cap \eta^-)\ =\ \vol(n \Delta_n)\ =\ \frac{n^n}{n!}\,,
  \end{split}
\end{align}
where the first inequality follows from (\ref{eq:1}). Combining this with (\ref{det}) implies the
desired inequality $\vol(K) \le \vol(\phi(K)) \le (n+1)^n/n!$.

\smallskip

Now, let us assume that $\vol(K)=(n+1)^n/n!$. This implies that all inequalities in (\ref{det}) and
(\ref{eq:2}) are in fact equalities.  The equations in (\ref{eq:2}) imply $\R^n_{\ge 0} \cap
\eta^-=\phi(K) \cap \eta^-$. Hence, $\phi(v) = \0 \in \R^n_{\ge0}\cap \eta^-$ yields $v \in K$. Since
$v$ was chosen arbitrarily we get $K=P$.  Moreover, equality in (\ref{det}) implies that
$|\det(\phi)|=1$, so $\phi$ is an affine lattice transformation of $\Z^n$. Thus, $P$ and $\phi(P)$
are unimodularly equivalent. In particular, since $\0$ is a lattice point, also $v$ is a
lattice point. Again, since $v$ was arbitrary, $P$ is a lattice polytope, and therefore reflexive
(recall that by assumption $P^*=Q$ is a lattice polytope).

We will use the following result on reflexive polytopes contained in the paper \cite{Nil06} of the first author:

\begin{lemma}\label{lem:nil06}
  Let $S$ be a reflexive polytope with barycenter in the origin. If $m$ is a relatively interior
  lattice point of a facet of $S$, then $-m$ is again in the relative interior of a facet of $S$.
\end{lemma}
\begin{proof}
 This is precisely the implication from ${\mathrm i.}$ to (a) in part (2) of Theorem 5.2 in \cite{Nil06}. 
In that paper $P$ being semisimple means that the set of roots is centrally-symmetric; and the set of roots 
is the precisely the set of lattice points in the relative interior of facets of $P$, see Definition 4.1 in \cite{Nil06} and the paragraph before.
\end{proof}

It remains to consider the following situation: We have an $n$-dimensional polytope $R$ (which is
the polytope $\phi(P)$ above) with $\vol(R)=\vol((n+1)\Delta_n)$ such that
\begin{enumerate}
\item the barycenter $b_R$ equals $\1$.
\item $n \Delta_n \subseteq R$ and they share the vertex cone at $\0$. Thus $e_1, \ldots, e_n$ are
  primitive inner facet normals.
\item $R-\1$ is a reflexive polytope.
\end{enumerate}

\begin{lemma} In this situation, we also have the following condition:
{\rm \begin{enumerate}\setcounter{enumi}{3}
  \item for any $i=1,\ldots,n$, the point $\1+e_i$ is contained in the relative
    interior of a facet of $R$.
\end{enumerate}}
\end{lemma}
\begin{proof} We may assume $i=1$. Since $n \Delta_n \subseteq R$, we see that $\1-e_1$ is in the
  relative interior of the facet $F := \{x \in R \;:\; \pro{e_1}{x}=0\}$. Hence, we can apply
  Lemma~\ref{lem:nil06} to $R-\1$ and $m=(\1-e_1)-\1 = -e_1$ to see that $e_1$ is in the relative
  interior of a facet of $R-\1$. This proves the claim.
\end{proof}
We will show that $R=(n+1)\Delta_n$.  For $i=1,\ldots,n$, let us denote by $G_i$ the unique facet of
$R$ containing $m_i := \1+e_i$. Furthermore, let $u_i$ be the primitive {\em outer} normal of
$G_i$. Translating a polytope does not change the normal fan, so, as $R-\1$ is reflexive, we obtain
\begin{align}
  \pro{u_i}{m_i-\1}\ & =\ \pro{u_i}{e_i} \ =\ 1 & \text{and}&&   \pro{u_i}{m_j-\1}\ & =\ \pro{u_i}{e_j} \ \le\ 1\quad \text{for $j\ne i$} \label{eq:eq2}
\end{align}
Let $u_1:=(c_1,\ldots,c_n)$.  Then (\ref{eq:eq2}) implies $c_1=1$ and $c_j\le 1$ for $j=2, \ldots,
n$. If $c_j=1$ for some $j\ge 2$, then $m_j \in G_1$, and thus $G_j=G_1$.

By assumption we have $n \Delta_n \subseteq R$, so $\pro{u_1}{n e_1 - \1} \le 1$. Thus, for any $j\ge 2$
\begin{align*}
  n\ &=\ \pro{u_1}{n\cdot e_1}\ \le\ \pro{u_1}{\1}+1\ =\ 2 + c_2 + c_3 + \cdots + c_n\\
  &\le\ 1+c_j+(n-1)\ =\ n+c_j\,.
\end{align*}
This implies that $c_j\ge 0$, so $c_j\in\{0,1\}$ for $j=2, \ldots, n$. Furthermore, if $c_j=0$ for
some $j$, then $c_k=1$ for all $k$ except $j$.

Now assume that indeed there is some $j$ with $c_j=0$, say $j=2$. Then $u_1=\1-e_2$. We have
observed above that $c_j=1$ implies $G_j=G_1$, hence $G_1=G_3=\cdots=G_n$. We consider $u_2:=(d_1,
\ldots, d_n)$. Similarly as for $u_1$ we obtain $d_2=1$ and $d_j\in\{0,1\}$ for $j\ne 2$.  Since
$m_2=\1+e_2 \not\in G_1$ (because of $\pro{u_1}{m_2-\1}=0<1$), we have $G_2 \not= G_j$ for $j \in
\{1,3,\ldots,n\}$. Hence, again by the above remark, we must have $d_k \not=1$ for any
$k=1,3,\ldots,n$.  Therefore, $u_2=e_2$, so $\pro{u_1}{m_1}=n$ and $\pro{u_2}{m_2}=2$ yields
\[R \subseteq \{x \in \R_{\ge 0}^n \;:\; \pro{\1-e_2}{x}\le n,\; \pro{e_2}{x}
\le 2\}.\] Here, the right polytope is unimodularly equivalent to $(n \Delta_{n-1}) \times [0,2]$.
Therefore,
\[\frac{(n+1)^n}{n!} = \vol(R) \le \frac{2 n^{n-1}}{(n-1)!},\]
which implies $(n+1)^n \le 2 n^n$, a contradiction for $n \ge 2$.\\

Therefore, $c_j=1$ for $j=2,\ldots, n$. Hence, $u_1=(1,\ldots,1)$. So, 
\[R \subseteq \{x \in \R_{\ge 0}^n \;:\; \pro{\1}{x}\le n+1\} = (n+1) \Delta_n.\] Since
both polytopes have the same volume, we get equality.  $\hfill\qed$

\begin{remark}
  As was pointed out to us by Szarek after this paper has appeared as a preprint, the proof can be
  substantially simplified by using a characterization of the equality case in Gr\"unbaum's
  inequality. In the notation of the proof: {\em inequality~\eqref{eq:1} holds with equality if and
    only if $K$ is a pyramid over a base spanning an affine hyperplane parallel to $h$, where $h$ is
    the affine hyperplane bounding the half-space $H$}.

  Here is how one applies this result in the equality case. Let $R=\phi(P)$ again with
  $\vol(R)=\vol((n+1)\Delta_n)$, barycenter $b_R=\1$, and $n \Delta_n \subseteq R$. From the
  previous result we deduce that $R$ is a pyramid with apex $\0$ and whose base facet has outer
  facet normal $\1$. Since the $n$ coordinate planes are supporting facet hyperplanes of $R$, this
  implies that $R$ must be a simplex of the form $\lambda \Delta_n$ for some $\lambda \ge n$.  From
  $\vol(R)=\vol((n+1)\Delta_n)$ we get $\lambda=n+1$, as desired.

  At the time this paper was published as a preprint, the proof of the equality characterization of
  Gr\"unbaum's inequality was not available in the literature. Morever, the main result in Gr\"unbaum's original paper \cite{Gru60} is stated slightly differently, and so a characterization of 
the equality case in that setting wouldn't be directly applicable to our situation. In 2003, Szarek found an alternative proof of Gr\"unbaum's inequality that allows to deal more directly with the case of equality. He made his proof now publicly available~\cite{Sza13}.\label{szarek}\end{remark}

\section{Generalized Ehrhart's conjecture and greatest lower bound on the Ricci curvature}

For a general Fano manifold $X$ one can define
\begin{align*}
  R(X)\ :=\ \sup(t\in[0,1]\ :\ \mathrm{Ric}\ \omega \ge t \omega \ \text{for all K\"ahler metrics }
  \omega\in c_1(X))\,,
\end{align*}
giving a ``greatest lower bound on the Ricci curvature''.  See \cite{Li11,Sze11} for background and
motivation. In Theorem~2.9 of \cite{BB12}, Berman and Berndtsson gave (in more generality) the
following generalization of Theorem~\ref{thm-bb} in the case that there is no K\"ahler-Einstein
metric:
\begin{theorem}[Berman, Berndtsson '12]
  Let $X$ be a toric Fano manifold. Then the first Chern class $c_{1}(X)$ satisfies the following
  upper bound
\[
c_{1}(X)^{n}\leq\left(\frac{n+1}{R(X)}\right)^{n}
\]
\label{thm-bb-ricci}
\end{theorem}

As a convex-geometric analogue, Li~\cite{Li11} introduced the following invariant for a convex body $K$.
\begin{definition} Let $K \subseteq \R^n$ be an $n$-dimensional convex body containing $\0$ in its
  interior. Let $b_K$ be its barycenter.  If $b_K=\0$, then we define $R(K) := 1$. Otherwise, let
  $x_K$ be the intersection of $b_K - \R_{\ge0} b_K$ with the boundary of $K$. Then $R(K)$ is
  defined as the distance of $x_K$ from $0$ divided by the distance of $x_K$ from $b_K$. In
  particular, $R(K) < 1$ in this case.
\end{definition}
Let $Q$ be an $n$-dimensional Fano polytope with associated toric variety $X_Q$.  In this
case Li~\cite{Li11} showed that for smooth $X_Q$ the two definitions of $R$ coincide, i.e.\
$R(X)=R(Q^*)$.

Taking this invariant $R(K)$ into account, Xiaodong Wang suggested a natural generalization of Ehrhart's conjecture
for convex bodies with arbitrary barycenter.
\begin{conj} Let $K \subseteq \R^n$ be an $n$-dimensional convex body. If $K$ contains only the
  origin as an interior lattice point, then
\[\vol(K) \leq \frac{(n+1)^n}{n!\, R(K)^n},\]
where equality holds if and only if $K$ is unimodularly equivalent to $(n+1) \Delta_n$.
\label{ehrhart-general}
\end{conj}
We would like to remark that this apparent generalization is actually equivalent to Ehrhart's
conjecture.
\begin{prop}
  Conjecture~\ref{ehrhart-general} is equivalent to Conjecture~\ref{ehrhart}.
\label{equiv}
\end{prop}
\begin{proof}
  Let $K \subseteq \R^n$ be an $n$-dimensional convex body such that $K$ contains only the origin as
  an interior lattice point. Let $r := R(K) < 1$, and $b_K$ the barycenter of $K$.

  We define $K' := r(K-b_K)$. We claim $K' \subseteq K$. For this, let $u \in \R^n$, $c \in \R$ such
  that $\pro{u}{x} \le c$ for any $x \in K$. Then $\pro{u}{r(x-b_K)} \le r(c-\pro{u}{b_K})$. It
  suffices to show that $r(c-\pro{u}{b_K}) \le c$. However, this is equivalent to $\pro{u}{(r/(r-1))
    b_K} \le c$ which holds by our assumption: $(r/(r-1)) b_K = x_K \in K$.

  Therefore, $K'$ is an $n$-dimensional convex body with barycenter $\0$ which contains no other
  interior lattice points.  Since $\vol(K') = r^n \vol(K)$, we have $\vol(K') \le (n+1)^n/n!$ if and
  only if $\vol(K) \le (n+1)^n/(n! r^n)$. It remains to consider the equality case.  Assume
  $K'=(n+1) \Delta_n$. Since $K \supseteq K'$ and any facet of $K'$ contains a lattice point in its
  relative interior, we have $K=K'$.
\end{proof}
From the proof we get the following generalization of Theorem~\ref{main}.
\begin{corollary}
  The conclusion of Conjecture~\ref{ehrhart-general} holds for any $n$-dimensional convex body $K \subseteq \R^n$ that
  is contained in the dual polytope of a lattice polytope.
\end{corollary}

In particular, this implies a convex-geometric proof of Theorem~\ref{thm-bb-ricci}. The following
consequence seems to be new:

\begin{corollary}
  In the situation of Theorem~\ref{thm-bb-ricci}, $c_{1}(X)^{n}=\left(\frac{n+1}{R(X)}\right)^{n}$
  if and only if $X \cong \P^n$.
\end{corollary}
\begin{proof}
Recall that $c_1(X)^n = n!\, \vol(Q^*)$ and $R(X)=R(Q^*)$. Hence, if equality in Theorem~\ref{thm-bb-ricci} holds, then the previous corollary implies that $Q^*$ is unimodularly equivalent to $(n+1) \Delta_n$, thus $X \cong \P^n$. On the other hand, if $X \cong \P^n$, then $c_1(X)^n=(n+1)^n$ and $R(X)=1$.\end{proof}

Finally, let us remark that Debarre showed in \cite{Deb01} on p.139 that for a general toric Fano
$n$-fold $X$ there is no polynomial bound on $\sqrt[n]{c_{1}(X)^{n}}$. In particular, there are
examples of toric Fano $n$-folds $X_n$ such that $R(X_n) \to 0$ for $n \to \infty$.

\end{document}